\documentclass[a4paper,12pt]{article}
\usepackage{amsmath,amsthm,amsfonts,amssymb,enumerate,calc,graphicx,iwona,float}
\usepackage[colorlinks=true,citecolor=black,linkcolor=black,urlcolor=blue]{hyperref}
\usepackage[tmargin=30mm,bmargin=30mm,rmargin=35mm,lmargin=35mm]{geometry}
\usepackage{microtype}
\usepackage{todonotes}
\usepackage[noabbrev,capitalise]{cleveref}
\crefname{lem}{Lemma}{Lemmas}
\crefname{thm}{Theorem}{Theorems}
\crefname{prop}{Proposition}{Propositions}

\graphicspath{{figures/}}
\renewcommand{\baselinestretch}{1.2}
\usepackage[hang]{footmisc} 

\renewcommand{\thefootnote}{\fnsymbol{footnote}}	
\allowdisplaybreaks
\newcommand\DateFootnote{
\begingroup
\renewcommand\thefootnote{}
\footnote{\today}
\setcounter{footnote}{0}
\vspace*{-3ex}
\endgroup}
\makeatletter
\renewcommand\section{\@startsection {section}{1}{\z@}{-3ex \@plus -1ex \@minus -.2ex}{2ex \@plus.2ex}{\normalfont\large\bfseries}}
\renewcommand\subsection{\@startsection{subsection}{2}{\z@}{-2.5ex\@plus -1ex \@minus -.2ex}{1.5ex \@plus .2ex}{\normalfont\normalsize\bfseries}}
\renewcommand\subsubsection{\@startsection{subsubsection}{3}{\z@}{-2ex\@plus -1ex \@minus -.2ex}{1ex \@plus .2ex}{\normalfont\normalsize\bfseries}}
 \renewcommand\paragraph{\@startsection{paragraph}{4}{\z@}{1.5ex \@plus.5ex \@minus.2ex}{-1em}{\normalfont\normalsize\bfseries}}
\renewcommand\subparagraph{\@startsection{subparagraph}{5}{\parindent}  {1.5ex \@plus.5ex \@minus .2ex}  {-1em} {\normalfont\normalsize\bfseries}}
\makeatother
\renewcommand{\thefootnote}{\fnsymbol{footnote}}	
\newcommand{\arXiv}[1]{arXiv:\,\href{http://arxiv.org/abs/#1}{#1}}

\newcommand{\ceil}[1]{\lceil{#1}\rceil}
\newcommand{\floor}[1]{\lfloor{#1}\rfloor}

\renewcommand{\geq}{\geqslant}
\renewcommand{\leq}{\leqslant}

\DeclareMathOperator{\dist}{dist}

\theoremstyle{plain}
\newtheorem{theorem}{Theorem}
\newtheorem{lemma}[theorem]{Lemma}
\newtheorem{corollary}[theorem]{Corollary}
\newtheorem{proposition}[theorem]{Proposition}

\theoremstyle{definition}

\begin{document}

{\Large\bfseries\boldmath\scshape Track Layouts, Layered Path Decompositions, \\and Leveled Planarity\footnotemark[1]}

\medskip
Michael J. Bannister\footnotemark[4] \quad
William E. Devanny\footnotemark[5] \quad
Vida Dujmovi{\'c}\footnotemark[2] \\
David Eppstein\footnotemark[5] \quad
David R. Wood\footnotemark[3]

\DateFootnote

\footnotetext[1]{A preliminary version of this paper entitled ``Track Layout is Hard'' was published in 
\emph{Proc.\ of  24th International Symp. on Graph Drawing and Network Visualization} (GD '16), 
Lecture Notes in Computer Science 9801:499--510, Springer, 2016.}

\footnotetext[4]{Department of Mathematics and Computer Science, Santa Clara University, California, USA, \texttt{mbannister@fastmail.fm}. Supported in part by NSF grant CCF-1228639.}

\footnotetext[5]{Department of Computer Science, University of California, Irvine, California, USA, \texttt{\{wdevanny,eppstein\}@uci.edu}. David Eppstein was supported in part by NSF grant  CCF-1228639. William E. Devanny was supported by an NSF Graduate Research Fellowship under grant DGE-1321846. }

\footnotetext[2]{School of Computer Science and Electrical Engineering, University of Ottawa, Ottawa, Canada, \texttt{vida.dujmovic@uottawa.ca}. Supported by NSERC and the Ministry of Research and Innovation, Government of Ontario, Canada.}

\footnotetext[3]{School of Mathematical Sciences, Monash University, Melbourne, Australia, \texttt{david.wood@monash.edu}. Supported by the Australian Research Council.}

\emph{Abstract.} We investigate two types of graph layouts, track layouts and layered path decompositions, and the relations between their associated parameters track-number and layered pathwidth. We use these two types of layouts to characterize leveled planar graphs, which are the graphs with planar leveled drawings with no dummy vertices. It follows from the known NP-completeness of leveled planarity that track-number and layered pathwidth are also NP-complete, even for the smallest constant parameter values that make these parameters nontrivial. We prove that the graphs with bounded layered pathwidth include outerplanar graphs, Halin graphs, and squaregraphs, but that (despite having bounded track-number) series-parallel graphs do not have bounded layered pathwidth. Finally, we investigate the parameterized complexity of these layouts, showing that past methods used for book layouts do not work to parameterize the problem by treewidth or almost-tree number but that the problem is (non-uniformly) fixed-parameter tractable for tree-depth.


\newpage

\section{Introduction}

A \emph{track layout} of a graph is a partition of its vertex set into sequences, called \emph{tracks}, such that the vertices in each track form an independent set and the edges between each pair of tracks form a non-crossing set. The \emph{track-number} of a graph is the minimum number of tracks in a track layout. Track layouts are connected with the existence of low-volume three-dimensional graph drawings: a graph has a three-dimensional drawing in an $O(1)\times O(1)\times O(n)$ grid if and only if it has track-number $O(1)$~\cite{DujMorWoo-SJC-05,DujWoo-AMS-04}. In this paper we show that track layouts are also related to a more abstract structure in graphs, a \emph{layered path decomposition}. This is a path decomposition together with a partition of the vertices of the graph into a sequence of layers, where the endpoints of each edge belong to a single layer or two consecutive layers. The \emph{width} of a layered path decomposition is the size of the largest intersection between a bag of the decomposition and a layer. The \emph{layered pathwidth} of a graph is the minimum width of a layered path decomposition. 

This paper first explores relationships between track layouts, layered pathwidth, and leveled planarity. A planar (undirected) graph is \emph{leveled planar} if it has a Sugiyama-style layered graph drawing with no crossings and no dummy vertices. This is a well studied model for planar graph drawing \cite{BM01,DETT99layered,HN13,SugTagTod-SMC-81}. We show that both track layouts and layered path decompositions can be used to characterize  leveled planar graphs. Specifically, we prove that leveled planar graphs are exactly the graphs with layered pathwidth at most 1, and are exactly the bipartite graphs with track-number at most 3 (see \cref{LeveledPlanarity}). Based on the known NP-completeness of testing leveled planarity~\cite{HeaRos-SJC-92}, it follows that testing whether the track-number is at most 3 is NP-complete. This solves an open problem from 2004~\cite{DujPorWoo-DMTCS-04}. In addition, it implies that testing whether the layered pathwidth is at most 1 is also NP-complete. In general, we prove that graphs of bounded layered pathwidth have bounded track-number  (see \cref{BeyondBipartiteness}). For track-number at most 3, we conjecture that the reverse is true, contrasting the fact that there exist graphs of track-number 4 and unbounded layered pathwidth.

Our second set of results show that many well-studied graph families are leveled planar or have bounded layered pathwidth (see \cref{sec:special}). In particular, we show that bipartite outerplanar graphs and squaregraphs have layered pathwidth~$1$ and are thus leveled planar. More generally, we prove that arbitrary outerplanar graphs and Halin graphs have layered pathwidth at most $2$, and unit disc graphs with bounded clique size have bounded layered pathwidth. On the other hand, series-parallel graphs (and even tree-apex graphs, a subclass of series-parallel graphs formed by adding a single vertex to a tree) have unbounded layered pathwidth, even though they do have bounded track-number. 

Finally, we study algorithmic aspects of leveled planarity, track-number, and layered pathwidth. We show that known methods of obtaining fixed-parameter tractable algorithms for other types of planar embedding, based on Courcelle's Theorem for treewidth~\cite{BanEpp-GD-14}, or on kernelization of the 2-core for $k$-almost-trees~\cite{BanEppSim-GD-13}, do not generalize to leveled planarity, track-number, or layered pathwidth. However, for any fixed bound on the tree-depth of the input graph, we give a non-constructive proof that these problems can be solved in linear time  (see \cref{ParameterizedComplexity}).

\section{Definitions}
\label{Definitions}

\subsection{Track layouts}

A \emph{$k$-track layout} of a graph is a partition of its vertex set into $k$ sequences, called \emph{tracks}, such that the vertices in each track form an independent set and the edges between each pair of tracks form a non-crossing set. This means that there are no edges $uv$ and $u'v'$ such that $u$ is before $u'$ in one track, but $v$ is after $v'$ in another track; such a pair of edges is said to form an \emph{X-crossing}. 

The \emph{track-number} of a graph $G$ is the minimum number of tracks in a track layout of $G$; this is finite, since the layout in which each vertex forms its own track is always non-crossing. The set of edges between two tracks form a forest of caterpillars (a forest in which the non-leaf vertices of each component induce a path); in particular, the graphs with track-number 1 are the independent sets, and the graphs with track-number 2 are the forests of caterpillars \cite{HS-UM-72}.

\subsection{Tree decompositions}

A \emph{tree-decomposition} of  a graph $G$ is given by a tree $T$ whose nodes index a collection $(B_x\subseteq V(G):x\in V(T))$ of sets of vertices in $G$ called  \emph{bags}, such that:
\begin{itemize}
\item For every edge $vw$ of $G$, some bag $B_x$ contains both $v$ and $w$, and 
\item For every vertex $v$ of $G$, the set $\{x\in V(T):v\in B_x\}$ induces a non-empty (connected) subtree of $T$.
\end{itemize}
The \emph{width} of a tree-decomposition is $\max_x |B_x|-1$, and the \emph{treewidth} of a graph $G$ is the minimum width of any tree decomposition of $G$. Treewidth was introduced (with a different but equivalent definition) by Halin~\cite{Halin76} and tree decompositions were introduced by Robertson and Seymour~\cite{RS-GraphMinorsII-JAlg86}.

A \emph{layering} of a graph is a partition of the vertices into a sequence of disjoint subsets (called \emph{layers}) such that each edge joins vertices in the same layer or consecutive layers. One way, but not the only way, to obtain a layering is the \emph{breadth first layering} in which we partition the vertices by their distances from a fixed starting vertex~\cite{DujMorWoo-FOCS-13,DMW17}. We emphasis that a layering does not specify an ordering of the vertices within each layer, so there is no notion of edge crossings in a layering. 


A \emph{layered tree decomposition} of a graph is a tree decomposition together with a layering. The \emph{layered width} of layered tree decomposition is the size of the largest intersection of a bag with a layer. The \emph{layered treewidth} of a graph $G$ is the minimum layered width of a tree-decomposition of $G$. Dujmovi\'c, Morin, and Wood~\cite{DujMorWoo-FOCS-13,DMW17} introduced layered treewidth and proved that every planar graph has layered treewidth at most $3$, that every graph with Euler genus $g$ has layered treewidth at most $2g+3$, and more generally that a minor-closed class has bounded layered treewidth if and only if it excludes some apex graph. Dujmovi\'c, Eppstein, and Wood~\cite{DEW15,DEW17} showed that layered treewidth is of interest well beyond minor-closed classes. For example, they proved that a graph embedded on a surface of Euler genus $g$ with at most $k$ crossings per edge has  layered treewidth $O((g+1)k)$. Analogous results were proved for map graphs defined with respect to any surface. Applications of layered treewidth include nonrepetitive graph colouring  \cite{DMW17}, queue layouts, track layouts and 3-dimensional graph drawings \cite{Duj15,DMW17}, book embeddings \cite{DF18}, intersection graph theory \cite{Shahrokhi13}, and graph structure theory \cite{DJMNW17}.

A \emph{path decomposition} is a tree decomposition where the underlying tree is a path~\cite{RobSey-JCTSB-83}. Thus, it can be thought of as a sequence of subsets of vertices (called \emph{bags}) such that each vertex belongs to a contiguous subsequence of bags and each two adjacent vertices have at least one bag in common.  \emph{Layered path decomposition} and \emph{layered pathwidth} are defined in an analogous way to layered tree decomposition and layered treewidth. So the \emph{layered pathwidth} of a graph $G$ is the minimum integer $k$ such that for some  path decomposition and layering of $G$, the intersection of each bag with each layer has at most $k$ vertices. The present paper is the first to consider layered path decompositions. 

Note that the layered pathwidth of a graph is at most one more than its pathwidth (just put every vertex in one layer). We can do slightly better as follows:

\begin{proposition}
Every graph with pathwidth $k$ has layered pathwidth at most $\ceil{\frac{k+1}{2}}$. 
\end{proposition}

\begin{proof}
It is well known and easily proved that every graph $G$ with pathwidth $k$ has a path  decomposition, given by a sequence of bags $B_1,\dots,B_n$, such that $|B_i| \leq k+1$ for all $i\in[n]$, and $|B_i\setminus B_{i-1}|=1$ and 
$B_{i-1}\not\subseteq B_{i}$ for all $i\in[2,n]$. We now construct a layering $(V_1,V_2)$ of $G$, such that $|B_i\cap V_1| \leq \ceil{\frac{k+1}{2}}$ and $B_i\cap V_2 \leq \floor{\frac{k+1}{2}}$ for all $i\in[n]$. First, put $\ceil{\frac{k+1}{2}}$ vertices of $B_1$ into $V_1$, and put the remaining vertices of $B_1$ into $V_2$. Now for $i=2,\dots,n$ perform the following step: let $w$ be the vertex in $B_i\setminus B_{i-1}$, and put $w$ on the same layer as a vertex in $B_{i-1}\setminus B_{i}$ (which exists since $B_{i-1}\not\subseteq B_i$).  Thus 
$|B_i\cap V_1| \leq |B_{i-1}\cap V_1| \leq \ceil{\frac{k+1}{2}}$ by induction. 
Similarly, $|B_i\cap V_2|\leq |B_{i-1}\cap V_2| \leq \floor{\frac{k+1}{2}}$. Thus $G$ has layered pathwidth at most 
 $\ceil{\frac{k+1}{2}}$. 
\end{proof}

More importantly, layered pathwidth might be much less than pathwidth. For example, it is well known that the pathwidth of the $n\times n$ grid equals $n$ (see \cite{HW17}), but the layered pathwidth of the $n\times n$ grid equals 1, as illustrated in \cref{fig:Grid}. 

\begin{figure}[h]
\centering\includegraphics{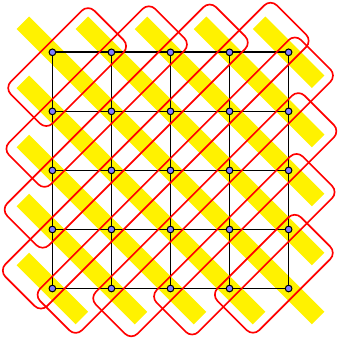}
\caption{Layered path decomposition of a grid with layered width 1, where layers are  yellow rectangles, and bags are red curves. The endpoints of each edge are in some bag, each vertex is in a consecutive set of bags, and each bag has at most one vertex in each layer. }
\label{fig:Grid}
\end{figure}

\subsection{Leveled planarity}

The class of \emph{leveled planar graphs} was introduced in 1992 by Heath and Rosenberg~\cite{HeaRos-SJC-92} in their study of queue layouts of graphs. A leveled planar drawing of a graph is a straight-line crossing-free drawing in the plane, such that the vertices are placed on a sequence of parallel lines (called \emph{levels}), where each edge joins vertices in two consecutive levels. Levels in a leveled planar drawing are numbered consecutively. These numbers are called \emph{level numbers}. A graph is \emph{leveled planar} if it has a leveled planar drawing. (Note that the ordering constraint on endpoints of pairs of edges between two tracks in a track layout is the same as the analogous constraint between two consecutive levels of a leveled planar drawing.)\

Note that leveled planar graphs correspond to  Sugiyama-style graph drawings~\cite{SugTagTod-SMC-81} that achieve perfect quality according to two of the most important quality measures for the drawing, the number of edge crossings~\cite{EadWor-Algo-94} and the number of dummy vertices~\cite{HeaNik-GD-01}. 

\cref{sec:special} shows that leveled planar graphs include several natural and well-studied classes of graphs, including the bipartite outerplanar graphs, squaregraphs, and dual graphs of arrangements of monotone curves. We characterize leveled planar graphs by both by their low track-number and their low layered pathwidth. This, together with the fact that recognizing leveled planar graphs is NP-complete~\cite{HeaRos-SJC-92}, will imply that testing whether the track-number or layered pathwidth is small is also NP-complete.

\section{Leveled planarity, track layouts, and layered pathwidth}
\label{LeveledPlanarity}

This section explores relationships between leveled planarity, track layouts, and layered pathwidth.


\subsection{Leveled planarity and track layouts}
\label{TrackLayouts}

\begin{figure}[t]
\centering\includegraphics[width=3in]{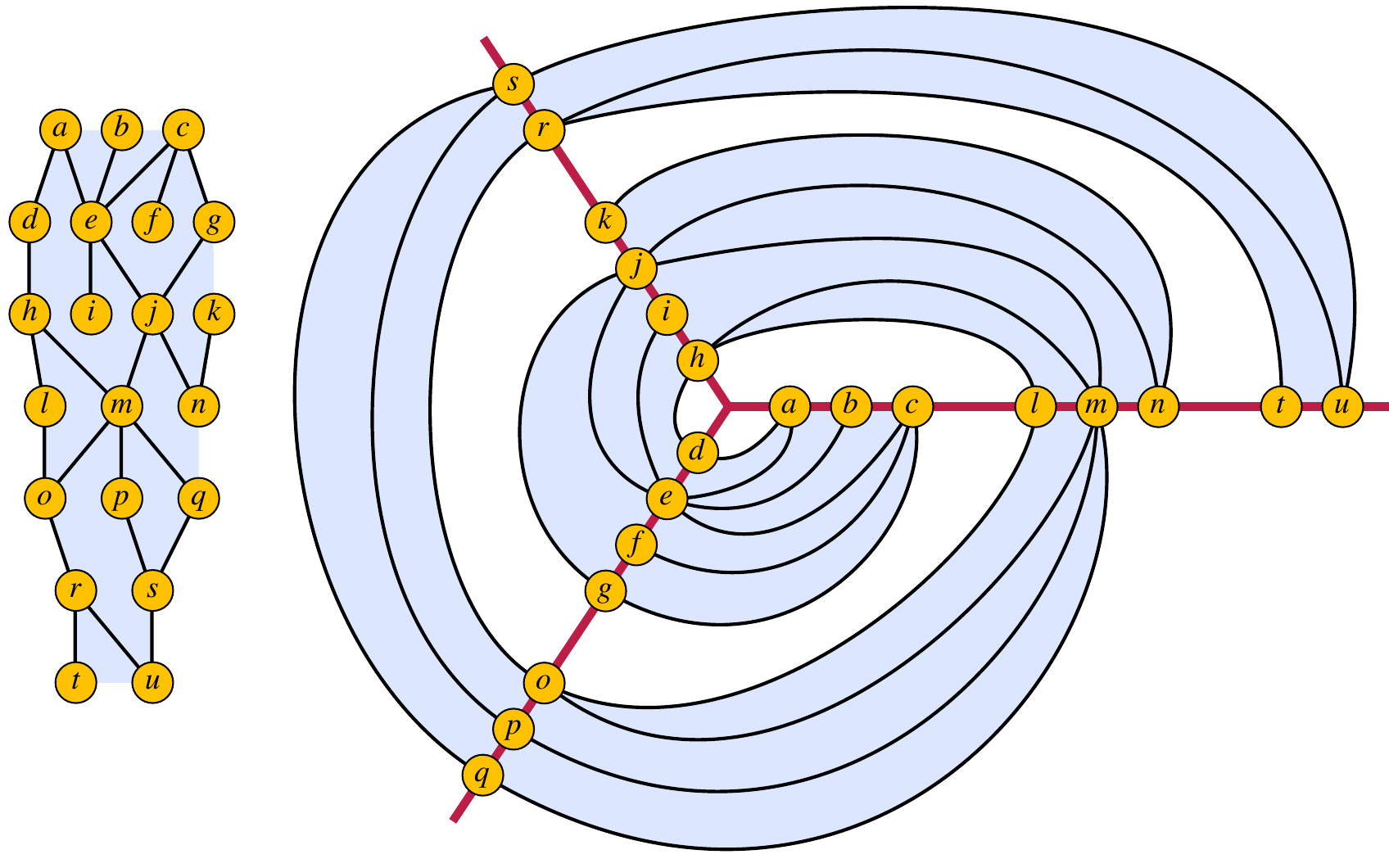}
\caption{Converting a leveled drawing to a 3-track layout}
\label{fig:Spiral}
\end{figure}

\begin{lemma}[implicit in \cite{FLW-JGAA-03}]
\label{lem:spiral}
Every leveled planar graph has a 3-track layout.
\end{lemma}

\begin{proof}
Assign the vertices of the graph to tracks according to their level number in the leveled drawing, modulo~3, as shown in \cref{fig:Spiral}. Within each track, order the levels by their level numbers, and then order the vertices within each level contiguously. Two edges that connect the same pair of levels cannot cross because of the chosen vertex ordering within the levels, and two edges that connect different pairs of levels but are mapped to the same pair of tracks cannot cross because of the ordering of the levels within the tracks.
\end{proof}

\cref{lem:spiral} can be interpreted as `wrapping' a leveled drawing on to 3 tracks; see \cite{DujPorWoo-DMTCS-04} for a more general wrapping lemma. As \cref{fig:Spiral} shows, a 3-track layout can also be interpreted geometrically, as a planar drawing in which the tracks are represented as three rays from the origin; it follows from this interpretation that 3-track graphs (and the weakly leveled planar graphs described in \cref{sec:weaklevel}) have universal point sets of size $O(n)$, consisting of $n$ points on each ray. However, for more than three tracks, a similar embedding of the tracks as rays in the plane would not lead to a planar drawing, because there is no requirement that edges of the graph connect only consecutive rays. Indeed, all graphs (for example, arbitrarily large complete graphs) have 4-track subdivisions~\cite{DujWoo-DMTCS-05}, and there are cubic expander graphs with 4-track layouts~\cite{DSW16}.

Define an \emph{arc} of an undirected graph $G$ to be a directed edge formed by orienting one of the edges of $G$. For a graph $G$ with a 3-track layout, define a function $\delta$ from the arcs of $G$ to $\pm 1$ as follows: if an arc $uv$ goes from track $i$ to track $i+1\pmod{3}$ (that is, if it is oriented clockwise in the planar embedding described above), let $\delta(uv)=+1$; otherwise (if it is oriented counterclockwise), let $\delta(uv)=-1$. For an oriented cycle $C$, we define (by abuse of notation) $\delta(C)=\sum_{uv\in C}\delta(uv)$.

\begin{lemma}
\label{lem:oddsum}
Let $C$ be a cycle embedded in a 3-track layout. Cyclically orient the edges of $C$. 
If $|C|$ is even then $\delta(C)=0$. If $|C|$ is odd then $|\delta(C)|=3$.
\end{lemma}

\begin{proof}
We proceed by induction on $|C|:=|V(C)|$. 
If $|C|=3$, then $C$ has one vertex on each track and $\delta(C)\in\{3,-3\}$. 
If $|C|=4$, then $C$ has two edges with $\delta=+1$ and two edges with $\delta=-1$, implying $\delta(C)=0$. 
Now assume that $|C|\geq 5$. Use the 3-track layout to embed $C$ in the plane as described in the proof of \cref{lem:spiral}, but with straight edges instead of the curved edges shown in \cref{fig:Spiral}.
As a planar polygon, $C$ has at least two ears, which are triangles formed by two of its edges that are empty of other vertices of $C$ (which may be found as the leaves in the tree formed as the dual graph of a triangulation of $C$). If one ear has the same sign of $\delta$ for both of the edges that form it, these edges must connect pairs of vertices that are the innermost  on their tracks.
Therefore, two such ears with same-sign edges could only exist if $C$ is a triangle.
For any longer cycle, let $uvw$ be an ear for which $\delta(uv)=-\delta(vw)$; thus edges $uv$ and $vw$ both connect the same two tracks, and (by the assumption that triangle $uvw$ is empty) $u$ and $w$ are consecutive in their track.
By deleting $v$ and merging $uw$ into a single vertex, we construct a cycle $C'$ with $|C'|=|C|-2\geq 3$, 
and a 3-track layout of $C'$ with $\delta(C')=\delta(C)$. The result follows by induction.
\end{proof}

The previous lemma can be restated in terms of winding number (see \cite{WindingNumber}). The \emph{winding number} of a closed curve $C$ in the plane around a given point $x$ is the number of times that $C$ travels counterclockwise around $x$. The contribution to the winding number of each edge $uv$ is $\frac{2\pi}{3}\,\delta(uv)$. So \cref{lem:oddsum} says that for an oriented cycle $C$ around the origin in a $3$-track representation of $C$ with three rays (as in \cref{fig:Spiral}), if $C$ is even then the winding number is 0, and if $C$ is odd then the winding number is 1. 
 
While \cref{lem:spiral} shows that a leveled planar drawing can be wrapped on to three tracks, we now use \cref{lem:oddsum} to show that a bipartite 3-track layout can be unwrapped to produce a leveled planar drawing. 

\begin{theorem}
\label{LeveledPlanarBipartite3Track}
A graph $G$ is leveled planar if and only if $G$ is bipartite and has a 3-track layout.
\end{theorem}

\begin{proof}
In one direction, if $G$ is leveled planar, then it is bipartite (with a coloring determined by the parity of the level numbers of the drawing) and has a 3-track layout by \cref{lem:spiral}.

In the other direction, suppose that $G$ is bipartite and has a 3-track layout.
We may assume without loss of generality that $G$ is connected, for otherwise we can draw each connected component of $G$ separately. Let $T$ be a spanning tree of~$G$. Root $T$ at an arbitrary vertex $r$ of $G$. 
For each vertex $v$ of $G$, let $T_v$ be the path from $r$ to $v$ in $T$, and let 
$$\ell(v):=\sum_{xy\in E(T_v)}\!\! \delta(xy).$$
Assign $w$ to level $\ell(v)$ in a leveled drawing of $G$. Note that levels might be negative. 
By construction, the endpoints of each edge of $T$ are assigned to consecutive levels. 

We now show that the same is true for each non-tree edge.  
Let $pq$ be an edge in $G-E(T)$. 
Let $v$ be the least common ancestor of $p$ and $q$ in $T$. 
Let $P$ be the path from $v$ to $p$ in $T$
Let $Q$ be the path from $v$ to $q$ in $T$. 
Let $Q'$ be the path from $q$ to $v$ in $T$. 
Let $C$ be the oriented cycle $(pqQ'P)$. 
Then 
\begin{align*}
\ell(p)-\ell(q) 
& = \sum_{xy\in E(T_p)} \!\!\delta(xy) \;-\; \sum_{xy\in E(T_q)} \!\!\delta(xy)\\
& = \sum_{xy\in E(P)} \!\!\delta(xy) \;-\; \sum_{xy\in E(Q)} \!\!\delta(xy).
\end{align*}
Now
\begin{align*}
\delta(C) 
& = \delta(pq) \;+\; \sum_{xy\in E(Q')} \!\!\delta(xy) \;+\; \sum_{xy\in E(P)} \!\!\delta(xy) \\
& = \delta(pq) \;-\; \sum_{xy\in E(Q)} \!\!\delta(xy) \;+\; \sum_{xy\in E(P)} \!\!\delta(xy) .
\end{align*}
Thus $\ell(p)-\ell(q) = \delta(C) - \delta(pq)$. 
Since $G$ is bipartite, $|C|$ is even. Thus $\delta(C)=0$ by \cref{lem:oddsum}.
Hence $\ell(p)-\ell(q) = - \delta(pq)$, which is $\pm 1$. 
Therefore the endpoints of each edge of $G$ are assigned to consecutive levels. 

Within each level of the drawing, the vertices all come from the same track, determined by the value of the level modulo~3. Assign the vertices to positions in left-to-right order on this level according to their ordering within this track. Then no two consecutive levels of the drawing can have crossing edges, because such a crossing would also be a crossing in the track layout. Therefore, this assignment of vertices to levels and to positions within these levels gives a leveled planar drawing of~$G$.
\end{proof}

\subsection{Leveled planarity and layered pathwidth}
\label{LayeredPathwidth}

The following lemma will allow us to build a layered path decomposition of a leveled planar graph greedily, one bag at a time.

\begin{lemma}
\label{lem:next-bag}
Let $G$ be a graph with a leveled planar drawing, and let $S$ be a subset of vertices of $G$ containing one vertex $s_i$ in each level~$i$ of the drawing. Suppose also that there exists at least one vertex in $S$ that is not the rightmost vertex in its level. Then there exists~$i$ such that $s_i$ is not rightmost in its level, and such that each neighbor of $s_i$ either belongs to $S$ or is positioned to the left of a vertex in $S$ within its level.
\end{lemma}

\begin{proof}
If every vertex in $S$ is either rightmost in its level or has no neighbor to the right of $S$, then we are done, for we may choose $s_i$ to be any vertex that is not rightmost in its level.

Otherwise, draw a directed graph $D$ whose vertices are the levels of $G$, with an edge from level $i$ to level $j$ ($j=i\pm 1)$ if  $s_i$ has a neighbor in level $j$ to the right of $s_j$. Then $D$ cannot contain edges in both directions between $s_i$ and $s_j$, for the corresponding edges in $G$ would necessarily cross, so $D$ must be a subgraph of an oriented path, and in particular must be a directed acyclic graph. By the assumption that at least one vertex in $S$ has a neighbor to the right of $S$, $D$ has at least one edge. Therefore, $D$ contains a vertex $i$ (a level of the drawing) that has incoming edges but that does not have any outgoing edges. For this level, $s_i$ is not rightmost in its level (else $i$ could have no incoming edges) but has no neighbors to the right of~$S$ (else $i$ would have an outgoing edge), as desired. 
\end{proof}

\begin{theorem}
\label{thm:lp2lpw1}
A graph $G$ is leveled planar if and only if it has layered pathwidth~1.
\end{theorem}

\begin{proof}
In one direction, suppose that $G$ has layered pathwidth 1. We can construct a leveled drawing of $G$ from its layered path decomposition, by using the layers of the layered path decomposition as the levels of the leveled drawing. Because the layered pathwidth is 1, any two vertices in the same level occur within disjoint intervals of the sequence of bags of the path decomposition of $G$, and so we can order the vertices within each level of the drawing by the ordering of their bags in the path decomposition. Each edge of $G$ joins two vertices in consecutive levels of this drawing (no edge joins two vertices in the same level because then the bag containing its endpoints would intersect that level in a set of size two or more). Draw each edge straight. Suppose that edges $vw$ and $xy$ cross, where $v$ is to the left of $x$ in one level, and $y$ is to the left of $w$ in a consecutive level. Some bag $B$ contains both $v$ and $w$. Every bag containing $y$ is to to the left of $B$, and every bag containing $x$ is to the right of $B$. Thus no bag contains both $x$ and $y$. This contradiction shows that no two edges cross. 


In the other direction (implicit in \cite[Lemma~1]{Dujmovic-etal-Algo08}), suppose that $G$ has a leveled planar drawing. We must show that this information can be used to find a layered path-decomposition of $G$ with width 1. For the layering of this layered path-decomposition, we use the sequence of levels of the drawing of $G$; because the drawing is assumed to have no dummy vertices, this satisfies the definitional requirement of a layering, that each edge connect vertices in the same layer or in two consecutive layers. For the path decomposition, we use a sequence of bags with one vertex per layer, this first of which is the set of vertices that are leftmost in their layer of the drawing. We construct this sequence of bags using a greedy algorithm from this starting bag, at each step using \cref{lem:next-bag} to find a vertex $v$ whose neighbors all belong to the present bag or earlier bags, and forming the next bag by replacing $v$ with the next vertex to the right of $v$ in the same level. In this way, by construction, each vertex belongs to a consecutive subsequence of bags. Each edge $vw$ has both neighbors in at least one bag (the first bag in the sequence to include the second of its two endpoints), because until the second endpoint has been introduced as part of the sequence of bags, the first endpoint cannot be replaced. Thus, we have a path decomposition with one vertex per layer, showing that the layered pathwidth is 1.
\end{proof}

\cref{LeveledPlanarBipartite3Track,thm:lp2lpw1} imply:

\begin{corollary}
\label{Equivalences}
The following are equivalent for a graph $G$:
\begin{itemize}
\item $G$ is leveled planar,
\item $G$ has layered pathwidth 1,
\item  $G$ is bipartite and has a 3-track layout. 
\end{itemize}
\end{corollary}

Note that \cref{Equivalences} is best possible, in the sense that there are leveled planar graphs that are not 2-track graphs (since 2-track graphs are simply forests of caterpillars \cite{HS-UM-72}). On the other hand, we conjecture that every 3-track graph (without restriction on bipartiteness) has bounded layered pathwidth. This conjecture would be false for 4-track graphs (see \cref{apexTree}).

\subsection{Weakly leveled planarity and layered pathwidth}
\label{sec:weaklevel}

A \emph{weakly leveled planar drawing} of a graph $G$ is a straight-line crossing-free drawing of $G$ in the plane, such that the vertices are placed on a sequence of parallel lines (again called \emph{levels}) and each edge joins two vertices that either belong to the same level or to consecutive levels. That is, we relax the definition of leveled planar drawings to allow edges between consecutive vertices on the same level.

\begin{theorem}
\label{WeakLayeredPathwidth}
If a graph $G$ has a weakly leveled planar drawing, then $G$ has layered pathwidth at most~$2$.
\end{theorem}

\begin{proof}
The proof is almost the same as that of \cref{thm:lp2lpw1}.
We use the sequence of levels of the drawing as the layers of a layered path-decomposition,
and form a sequence of bags with one vertex per layer, covering the edges that connect two consecutive layers. As in \cref{thm:lp2lpw1} we construct this sequence of bags greedily, using \cref{lem:next-bag} to find a vertex $v$ whose neighbors on adjacent levels all belong to the present bag or earlier bags. In the proof of \cref{thm:lp2lpw1}, the next bag was formed by replacing $v$ with the next vertex $w$ to the right of $v$ on the same level, but if we did that we might form a sequence of bags that did not include a bag containing both $vw$, violating the definition of a path decomposition if $v$ and $w$ are adjacent. Instead, we first add $w$ to the present bag, forming a bag whose intersection with $v$'s layer has two vertices, and then we form a second bag by removing $v$. The result is a path decomposition: every edge between consecutive levels is represented in at least one bag by the same reasoning as in the proof of \cref{thm:lp2lpw1}, and every edge with two endpoints on the same level is represented in at least one bag by construction.
Its largest intersection with a level has size two, so the layered pathwidth is~$2$.
\end{proof}


\subsection{Layered pathwidth and track-number}
\label{BeyondBipartiteness}

Dujmovi\'c, Morin, and Wood~\cite{DujMorWoo-SJC-05} proved that every graph with pathwidth $k$ has track-number at most $k+1$. Here we provide the following qualitative strengthening (since there are graph classes with bounded layered pathwidth and unbounded pathwidth). 

\begin{lemma}
\label{qntnpb}
Every graph with layered pathwidth $\ell$ has track-number at most $3\ell$. 
\end{lemma}

\begin{proof}
Let $B_1,B_2,\dots,B_n$ be a path decomposition of $G$ with layered width $\ell$. Let $(V_0,V_1,\dots,V_m)$ be the corresponding layering. Thus, each bag $B_i$ contains at most $\ell$ vertices in each layer $V_j$. Since $G[V_j]$ has pathwidth at most $\ell-1$, there is a proper colouring of $G[V_j]$ with colours $1,2,\dots,\ell$. 
For each vertex $v$ of $G$, let $b(v):=\min\{i:v\in B_i\}$ be the index of the leftmost bag containing $v$. 
For $0\leq j\leq m$ and $1\leq a\leq\ell$, let $V_{j,a}$ be the set of vertices in $V_j$ coloured $a$. 
Let $\preceq$ be the total order of $V_{j,a}$ defined by $v\prec w$ if and only if $b(v)< b(w)$. 
Clearly $\preceq$ is a total order. 

\begin{figure}[h]
\centering
\includegraphics[height=80mm]{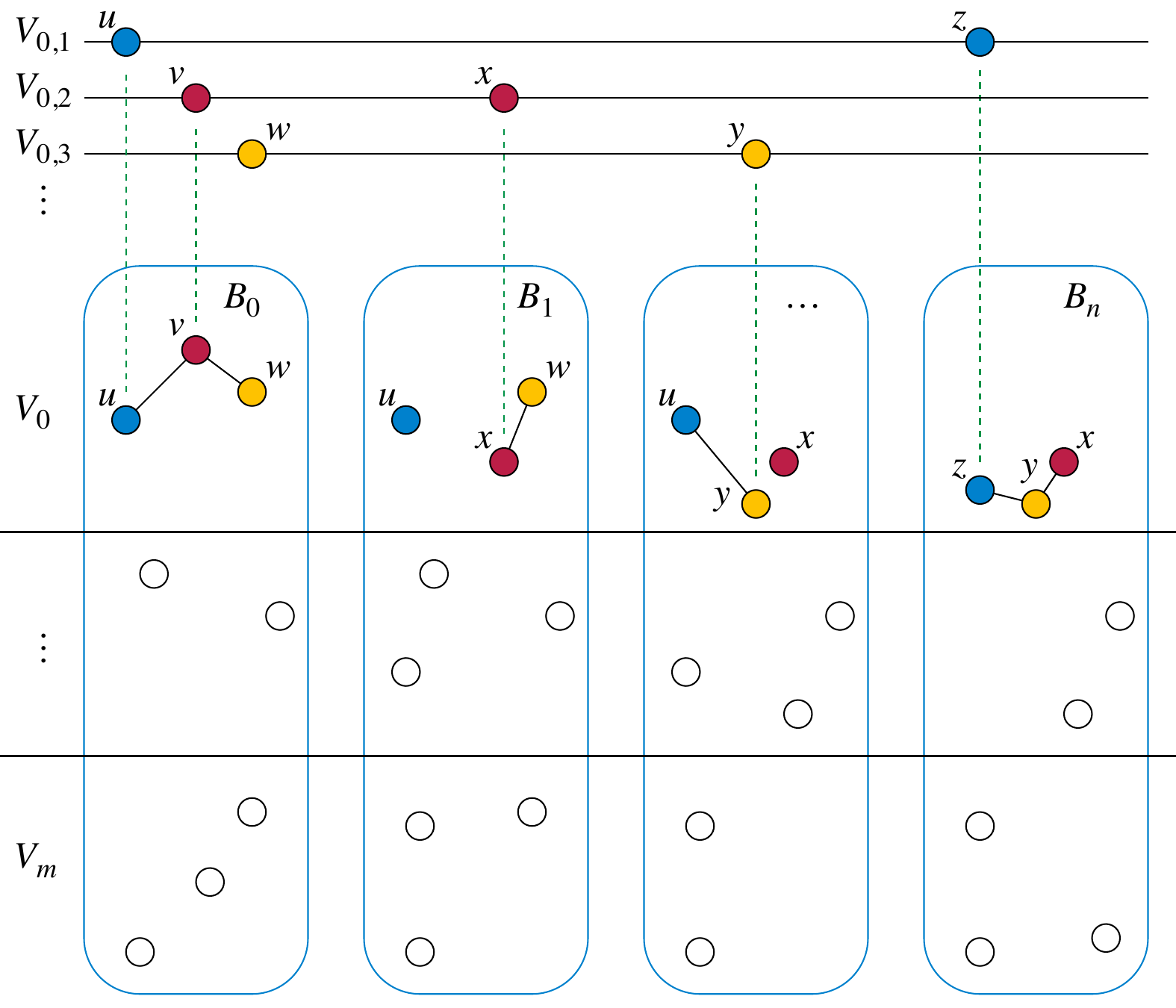}
\caption{\label{tracks-from-colors}
Construction of track layout in \cref{qntnpb}.}
\end{figure}

Since $G[V_j]$ is properly coloured, $V_{j,a}$ is a track. Suppose on the contrary that $vw$ and $xy$ form an X-crossing between $V_{j,a}$ and $V_{k,b}$, where $v\prec x$ in $V_{j,a}$ and $y\prec w$ in $V_{k,b}$. Without loss of generality, $b(w)\leq b(x)$. Since $y\prec w$ we have $b(y)<b(w)\leq b(x)$. Since $xy$ is an edge, $y\in B_{b(w)}$. Hence $y$ and $w$ are adjacent in $G[V_k]$, which is a contradiction since $y$ and $w$ are assigned the same colour. Therefore there is no X-crossing, and $\{V_{j,a}:0\leq j\leq m,1\leq a\leq \ell\}$ is a track layout of $G$. Since $(V_0,V_1,\dots,V_m)$ is a layering, if $vw$ is an edge of $G$ with $v\in V_{j,a}$ and $w\in V_{k,b}$, then $|j-k|\leq 1$. It follows from a result of Dujmovi\'c, Por and Wood~\cite[Lemma~6 with $s=1$]{DujPorWoo-DMTCS-04} that this track layout can be wrapped onto $3\ell$ tracks. In particular, as illustrated in \cref{wrap-tracks}, for $q\in\{0,1,2\}$ and $1\leq a\leq \ell$, let $W_{q,a}$ be the track $V_{q,a},V_{q+3,a},V_{q+6,a},\dots$. Then $\{W_{q,a}:q\in\{0,1,2\},1\leq a\leq \ell\}$ is a $3\ell$-track layout of $G$. 
\end{proof}

\begin{figure}[h]
\centering
\includegraphics[height=80mm]{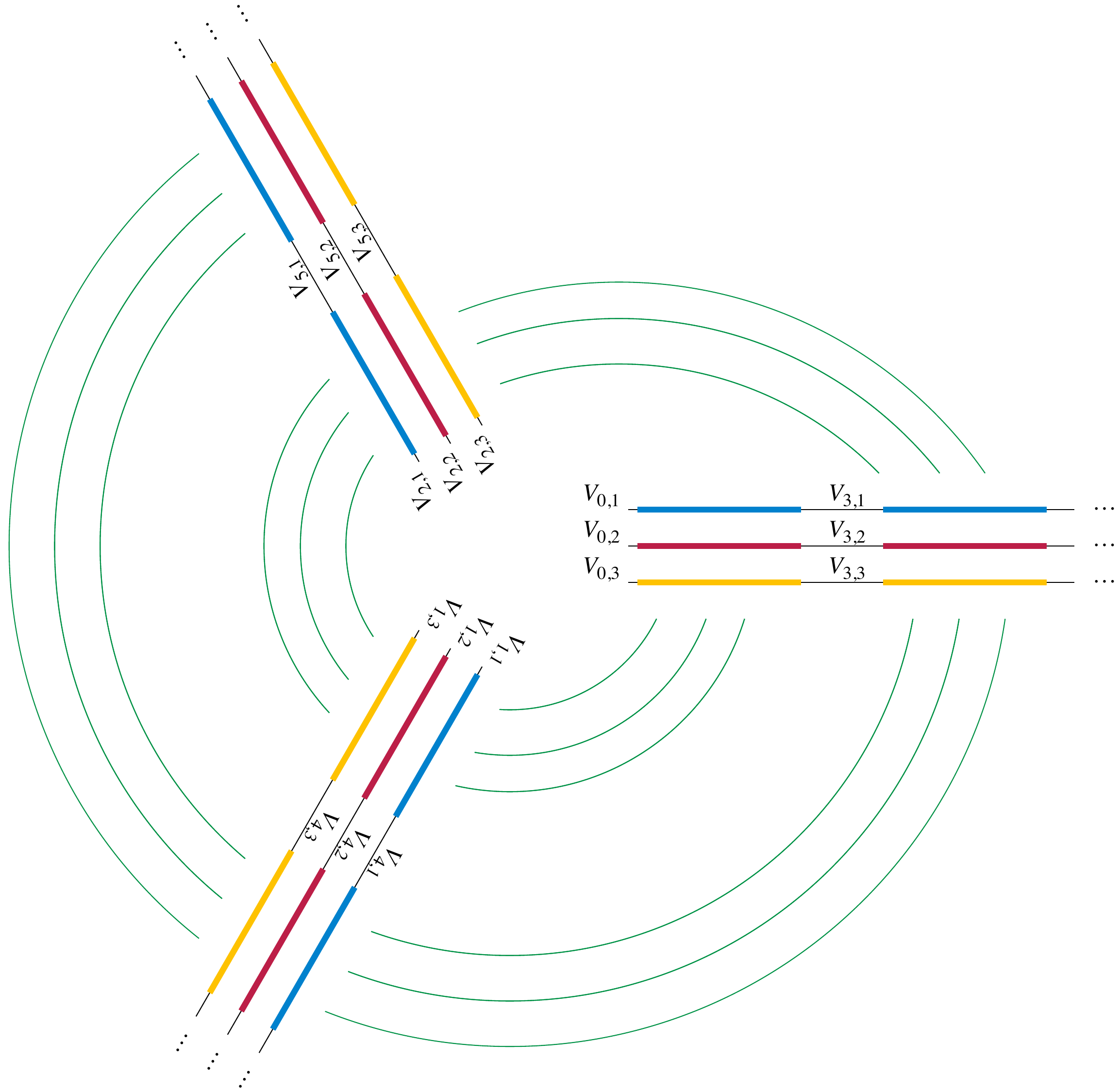}
\caption{\label{wrap-tracks}
Wrapping the track layout in  \cref{qntnpb}.}
\end{figure}

%
%
%
%
%
%

There is a natural connection between layered treewidth and layered pathwidth. 

\begin{lemma}
\label{log}
Every $n$-vertex graph $G$ with layered treewidth $\ell$ has layered pathwidth at most $\ell\log_3(2n+1)$. 
\end{lemma}

\begin{proof} 
Let $(B_x:x\in V(T))$ be a tree decomposition  of $G$ with layered width $\ell$.  That is, each bag $B_x$ contains at most $\ell$ vertices in some layering. If $B_x=B_y$ for some edge $xy\in E(T)$, then contracting $xy$ gives a tree decomposition with layered width $\ell$. Thus, we may assume that $B_x\neq B_y$ for each edge $xy\in E(T)$. It follows that $T$ has at most $n$ vertices. Scheffler~\cite{Sch92} proved that every $n$-vertex tree has pathwidth at most $\log_3(2n+1)$. Let $B_1,\dots,B_m$ be a path decomposition of $T$ with width $\log_3(2n+1)$. Let $B'_i:=\cup\{V_x:x\in B_i\}$. Then  $B'_1,\dots,B'_m$ is a path decomposition of $G$ with layered width at most $\ell\log_3(2n+1)$ (with respect to the initial layering). 
\end{proof}

\cref{qntnpb,log} imply the following result, which improves the constant factor in a result of Dujmovi\'c~\cite{Duj15}.

\begin{theorem}
Every $n$-vertex graph with layered treewidth $\ell$ has track-number at most $3\ell\log_3(2n+1)$.  
\end{theorem}

Note that, similar to layered treewidth \cite{DMW17}, layered pathwidth is also not a minor-closed parameter. For example,  it is easily seen that the $n \times n \times 2$ grid graph has layered pathwidth at most $3$, but $G$ contains a $K_n$ minor \cite{DW-NY11}, and $K_n$ has layered pathwidth $\lceil n/2 \rceil$.

\section{Special classes of graphs}
\label{sec:special}

Here we prove that particular graph families are leveled planar or weakly leveled planar. Our results are based on breadth-first layerings; we define a layering of a graph to be \emph{planar} if there exists a leveled planar drawing of the graph in which the levels of the drawing are the same as the layers of the layering; see \cref{fig:Examples} for examples.

\begin{figure}[t]
\centering\includegraphics[width=\textwidth]{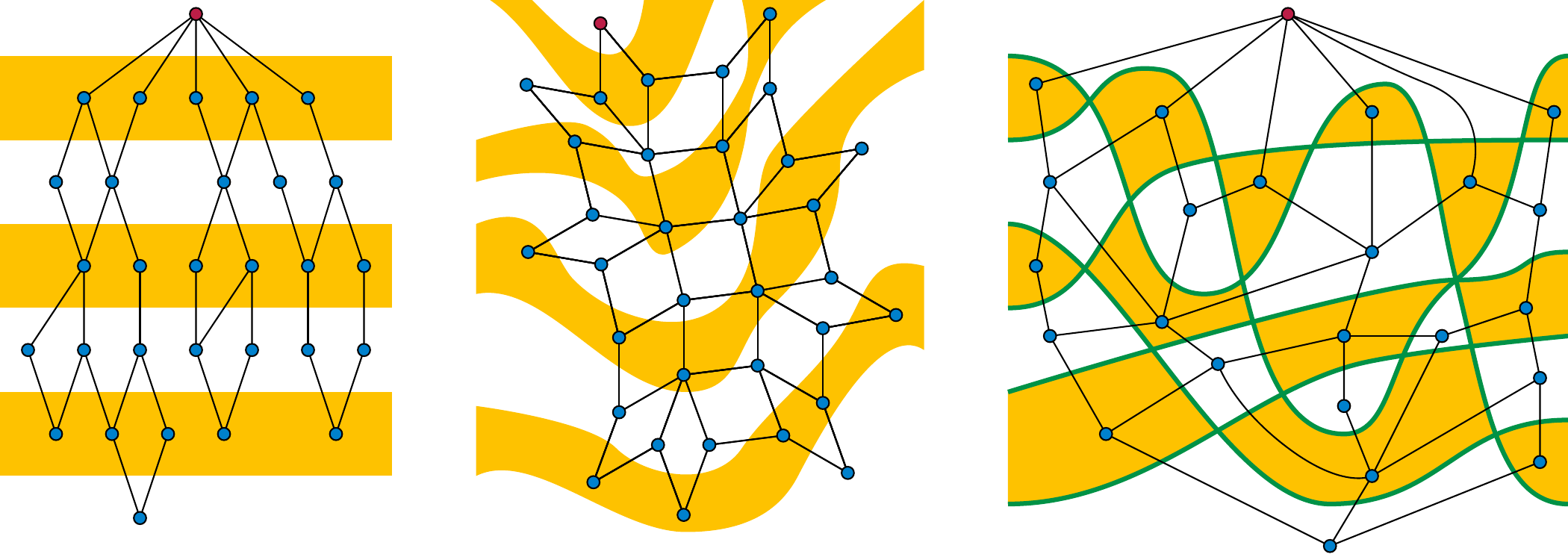}
\caption{Examples of graphs with planar breadth-first layerings (start vertex shown in red, and layering in yellow and white): left, a bipartite outerplanar graph (\cref{outerplanar}); center, a squaregraph (\cref{squaregraph}); and right, the dual graph of an arrangement of doubly-unbounded monotonic curves (\cref{monotone}).
\label{fig:Examples}}
\end{figure}

\subsection{Bipartite outerplanar graphs}

\begin{theorem}[implicit in \cite{FLW-JGAA-03}]
\label{outerplanar}
Every bipartite outerplanar graph is leveled planar. Every breadth first layering of such a graph $G$ gives a leveled planar drawing.
\end{theorem}

\begin{proof}
Let $v$ be the starting vertex of a breadth first layering. Then for each facial cycle $C$ of the outerplanar embedding of $G$, there must be a unique nearest neighbor in $C$ to $v$. For, if $v$ were nearest to distinct vertices $u$ and $w$ in $C$, then by bipartiteness these two vertices must be non-adjacent in $C$. In this case, the graph formed by $C$ together with the shortest paths from $v$ to $u$ and $w$ would contain a subdivision of $K_{2,3}$ (with $u$ and $w$ as the degree three vertices, two paths between them in $C$, and one more path between them through the shortest path tree rooted at $v$), an impossibility for an outerplanar graph. For the same reason, the distances in $v$ from this nearest neighbor or pair of nearest neighbors must increase monotonically in both directions around $C$ until reaching a unique farthest neighbor, because in the same way any non-monotonicity could be used to construct a subdivision of $K_{2,3}$.

Thus, each facial cycle of $G$ has a planar breadth first layering. The result follows from the fact that in a plane graph with an assignment of levels to the vertices, there is a planar drawing consistent with this level assignment and with the given embedding of the graph, if and only if every facial cycle of the given graph has a planar drawing consistent with the level assignment~\cite{AbeDemDem-GD-14}.
\end{proof}

\subsection{Squaregraphs}

A \emph{squaregraph} is defined to be a graph that has a planar embedding in which each bounded face is a $4$-cycle and each vertex either belongs to the unbounded face or has four or more incident edges. These graphs may also be characterized in various other ways, for instance as the dual graphs of hyperbolic line arrangements with no three mutually-intersecting lines~\cite{BanCheEpp-SJDM-10}.

\begin{theorem}
\label{squaregraph}
Every squaregraph $G$ is leveled planar. In fact, every breadth first layering of $G$ rooted at a vertex of the outerface gives a leveled planar drawing.
\end{theorem}

\begin{proof}
Because all their bounded faces are even-sided, squaregraphs are necessarily bipartite, so every choice of starting vertex gives a valid breadth first layering. Bandelt et al~\cite[Lemma 12.2]{BanCheEpp-SJDM-10} prove that, for every choice of starting vertex, we can add extra edges to the squaregraph to form a plane multigraph in which the added edges link each layer into a cycle, and in which these cycles are all nested within each other.

Now, choose the starting vertex $v$ to be a vertex of the outer face. Then each cycle added in this augmentation of $G$ contains an edge that separates $v$ from the unbounded face of the augmented graph. If we remove each such edge from the augmented graph, we break each cycle into a path in a consistent way, such that the path ordering within each layer matches the given planar embedding of~$G$.
\end{proof}

\subsection{Dual graphs of monotone curves}

\begin{theorem}
\label{monotone}
Let $A$ be a collection of finitely many $x$-monotone curves in the plane, such that any two curves intersect at finitely many crossing points, and the projection of $\bigcup A$ onto the $x$-axis covers the entire axis. Then the dual graph of the arrangement of the curves in $A$ is leveled planar, and there is a breadth first layering that gives a leveled planar drawing.
\end{theorem}

\begin{proof}
Each vertex of the dual graph corresponds to a connected component of the complement of $\bigcup A$; we call this the \emph{region} of the vertex. We may assign each vertex to a layer according to the number of curves in $A$ that pass above it; this is a breadth first layering starting from the vertex corresponding to the topmost (unbounded upward) connected component. No two vertices in the same layer have regions that project to overlapping subsets of the $x$-axis, so we may order the vertices within each layer according to the left-to-right ordering of these projections. This ordering is compatible with the planar embedding of the dual graph given by placing a representative point within each region and connecting each two adjacent regions by a curve crossing their shared boundary.
\end{proof}

\begin{figure}[t]
\centering\includegraphics[width=2in]{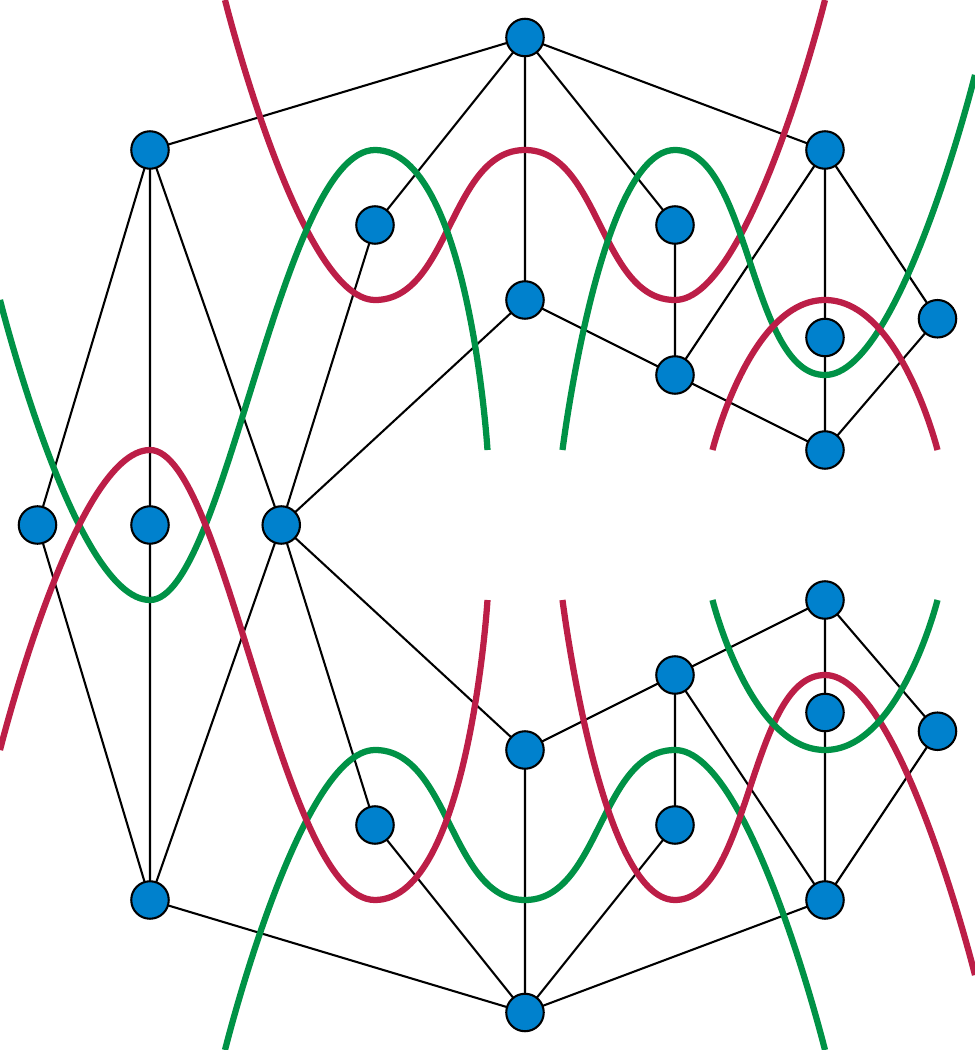}
\caption{An arrangement of monotone curves whose dual graph has no leveled planar drawing.}
\label{fig:BadCurves}
\end{figure}

\cref{fig:BadCurves} gives an example demonstrating that \cref{monotone} cannot be generalized to monotone curves whose projections do not cover the entire axis: it gives a family of monotone curves, all ending within the outer face of their arrangement, such that the dual graph of the arrangement is not leveled planar. The dual graph is made of multiple $K_{2,3}$ subgraphs, each of which must have the 2-vertex side of its bipartition drawn on two layers with the 3-vertex side of its bipartition in a single layer between them; thus, up to top-bottom reflection, there is only a single layering for this graph that could possibly be planar. However, this layering forced by the planarity of the individual $K_{2,3}$ subgraphs is not planar globally, because it forces one of the two arms of the graph (upper and lower right) to collide with the ``armpit'' where the other arm meets the body of the graph (left). The graph is drawn without crossings in the figure, but in a way that does not respect any layering of the graph. This example is also a series-parallel graph, and shows that \cref{outerplanar} cannot be generalized to bipartite series-parallel, treewidth-$2$, or $2$-outerplanar graphs: none of these classes of graphs is leveled planar.

\subsection{Outerplanar graphs}

\begin{theorem}[Felsner, Liotta, and Wismath~\cite{FLW-JGAA-03}]
\label{thm:op-weak-level}
Every outerplanar graph has a weakly leveled planar drawing.
\end{theorem}

Felsner et al.~\cite{FLW-JGAA-03} prove this result by a construction based on breadth-first search, using the BFS number and depth in the BFS tree as coordinates. Alternatively, \cref{thm:op-weak-level} can be proven by using induction on the number of triangular faces of a maximal outerplanar graph to show that each such graph has a layout in which, on each edge of the outer face, there is room to add one more triangle with its new vertex one level below the upper level of the previous triangle vertices. Felsner et al.{} wrapped such a drawing (as in \cref{fig:Spiral}) to produce an improper 3-track layout (allowing edges between consecutive vertices in a track) of any outerplanar graph. Dujmovi\'c et~al.~\cite{DujPorWoo-DMTCS-04} proved that every outerplanar graph has a (proper) 5-track layout. 

\cref{WeakLayeredPathwidth,thm:op-weak-level} imply:

\begin{corollary}
Every outerplanar graph has layered pathwidth at most~$2$.
\end{corollary}

\subsection{Halin graphs}
\begin{figure}[t]
\centering\includegraphics[width=2.5in]{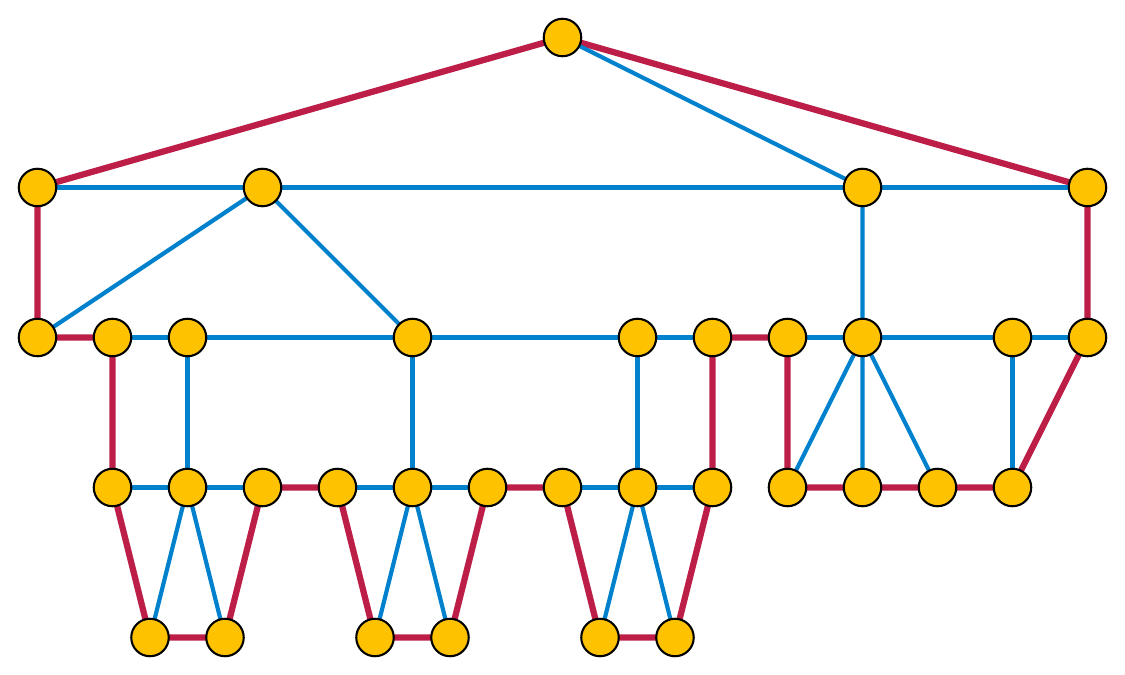}
\caption{A Halin graph drawn with all edges on a single level or consecutive levels}
\label{fig:LeveledHalin}
\end{figure}

Recall that a \emph{Halin graph}~\cite{Hal-CMA-71} is the graph formed from a tree with no degree-2 vertices, embedded in the plane, by connecting the leaves of the  tree by a cycle in the order given by the embedding. Di Giacomo and Meijer~\cite{DM-GD03} proved that every Halin graph has a 5-track layout, and described a Halin graph with track-number at least 4. As far as we are aware, it is open whether every Halin graph has a 4-track layout.

\begin{theorem}
\label{HalinWeak}
Every Halin graph has a weakly leveled planar drawing.
\end{theorem}

\begin{proof}
Choose an arbitrary leaf of the tree from which the Halin graph was constructed, to be the root of the tree. Then assign vertices to levels as follows: the root is assigned to level $0$. Then, in stage~$i$ of the assignment ($i=1,2,\dots$) we assign to level~$i$ the previously-unassigned nodes that are either children of nodes at level~$i-1$, or that belong to a path from such a child to its leftmost or rightmost leaf descendant.

This level assignment (depicted in \cref{fig:LeveledHalin}) is consistent with the given planar embedding of the tree and the Halin graph formed from the tree. Clearly, it embeds each two vertices that are adjacent in the tree to the same level or consecutive levels, because if one of the two adjacent vertices is assigned to level~$i$ in stage~$i$ but the other is not, then the second vertex will be one of the children of the first vertex assigned to level~$i+1$ in the next stage. Additionally, pairs of vertices that are adjacent in the outer cycle of the Halin graph must also belong to the same level or adjacent levels, because (with the exception of the two edges incident to the root, for which a similar argument is possible) each such pair of vertices must consist of the rightmost descendant of  one child and the leftmost descendant of the next child of the lowest common ancestor of the two nodes. The two children of the common ancestor can only be one level apart, and the same follows for their extremal descendants. 
\end{proof}

\cref{WeakLayeredPathwidth,HalinWeak} imply:

\begin{corollary}
Every Halin graph has layered pathwidth at most~$2$.
\end{corollary}

\subsection{Series-parallel and tree-apex graphs}

\begin{figure}[t]
\centering\includegraphics[width=3in]{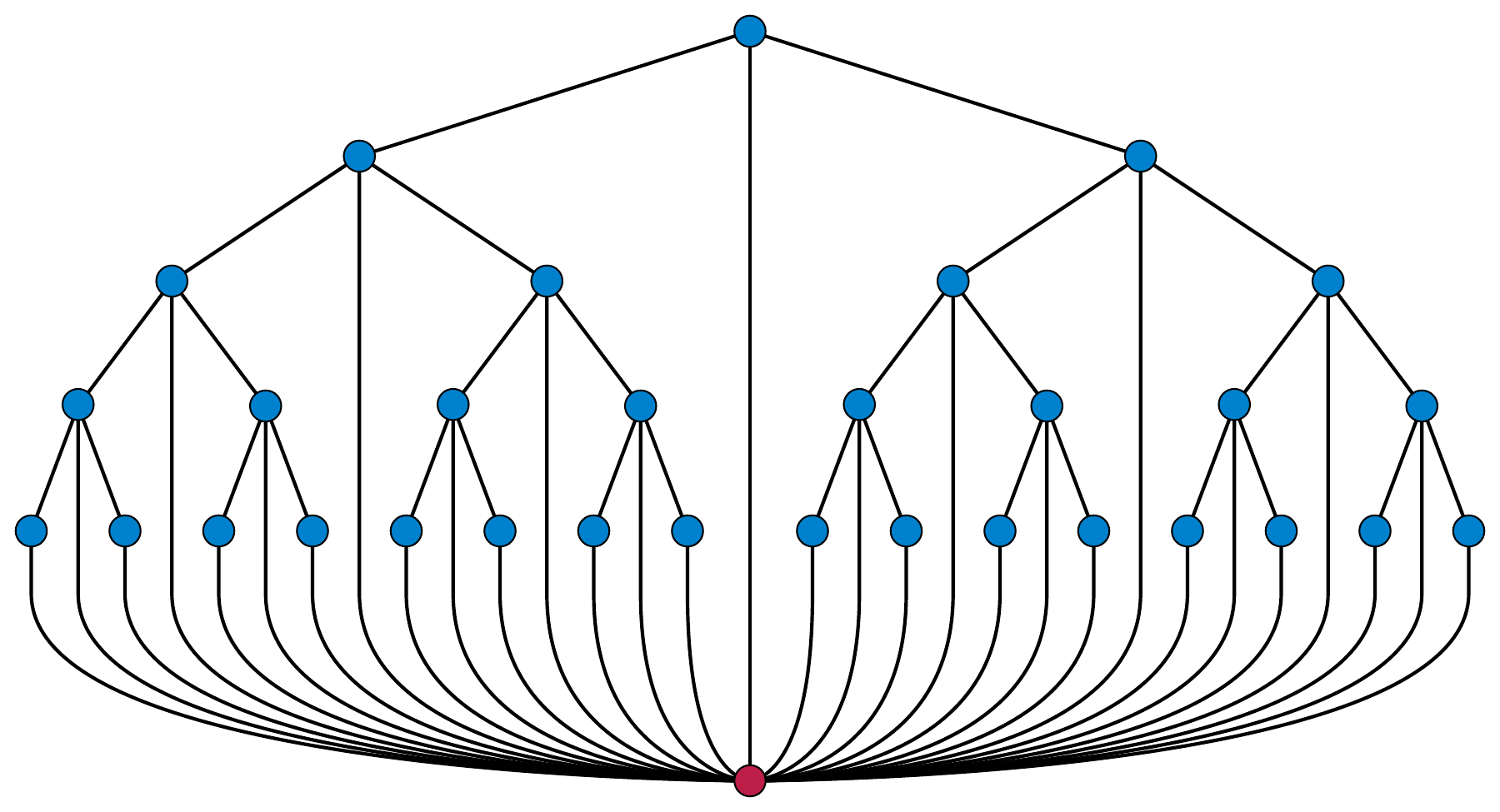}
\caption{Adding an apex to a complete binary tree produces a series-parallel graph with track-number 4 but arbitrarily large layered pathwidth.}
\label{fig:ApexTree}
\end{figure}

Although series-parallel graphs are in some sense intermediate in complexity between outerplanar graphs and Halin graphs (for instance, series-parallel graphs and outerplanar graphs have treewidth~2, whereas Halin graphs in general have treewidth~3), it is not true that every series-parallel graph has bounded layered pathwidth. In fact it is not true that every \emph{tree-apex} graph has bounded layered pathwidth, that is, a graph that can be obtained from a tree by adding a universal vertex that is adjacent to all other vertices (\cref{fig:ApexTree}). 

\begin{theorem}
\label{apexTree}
For every integer $h\geq 0$, there exists a series-parallel graph, in fact a tree-apex graph, that has a track-number at most $4$ and layered pathwidth  $\Theta(h)$.
\end{theorem}

\begin{proof}
Consider the tree-apex graph $G$ formed from a complete binary tree of height~$h$ by adding a universal vertex that is adjacent to all other vertices (\cref{fig:ApexTree}). $G$ is series-parallel, has track-number at most $4$ (since the tree has a  $3$-track layout), and has pathwidth $\Theta(h)$. $G$ does not have bounded layered pathwidth because the universal vertex forces every layering of this graph to use at most three layers (the one containing this vertex and at most two adjacent layers). Every path decomposition of $G$ has a bag with $\Omega(h)$ vertices, and the largest of the three intersections of this bag with a layer must also have $\Omega(h)$ vertices. Therefore, $G$ has layered pathwidth $\Omega(h)$.
\end{proof}

\subsection{Unit disc graphs}

For a set $P$ of points in the plane, the unit disc graph $G$ of $P$ has vertex set $P$, where $vw \in E(G)$ if and only if $\dist(v,w) \leq 1$.

\begin{theorem}
\label{UnitDiscGraphs}
Every unit disc graph with maximum  clique size $k$ has layered pathwidth at most $4k$.
\end{theorem}

\begin{proof}
Say vertex $v\in P$ has coordinates $(x(v),y(v))$.
Assume $y(v) \geq 0$ for all $v\in P$.
For each non-negative integer $j$, let $V_j := \{v \in V(G): j \leq y(v) < j+1\}$.
Then $V_0, V_1,\dots$ is a layering of $G$.
Say $V(G) = \{v_1,\dots,v_n\}$ where $x(v_1) \leq x(v_2) \leq \dots \leq x(v_n)$ .
For $i\in[1,n]$, let $B_i := \{ v \in V(G) : x(v_i) \leq x(v) \leq x(v_i)+1 \}$.
Observe that $B_1,\dots,B_n$ is a path decomposition of $G$.
Now consider the layered pathwidth.
Every vertex in $B_i \cap V_j$ lies in the unit square with bottom-left
corner $(x(v_i),j)$.
Partition this square into four subsquares of side length 1/2.
At least  $|B_i \cap V_j|/4$ vertices lie in one such subsquare, and
these vertices form a clique in $G$.
Thus $k \geq |B_i \cap V_j|/4$ and $|B_i \cap V_j| \leq 4k$.
\end{proof}

Of course, $k/2$ is a lower bound on the layered pathwidth in \cref{UnitDiscGraphs}.

\section{Parameterized complexity}
\label{ParameterizedComplexity}

A problem is \emph{uniformly fixed-parameter tractable} if there is an algorithm that solves it in polynomial time for any fixed value of the parameter. 
A problem is \emph{non-uniformly fixed-parameter tractable} if there is a collection of algorithms such that for each possible fixed value of the parameter one of the algorithms solves the problem in polynomial time. See~\cite{FPT-Book} for an introduction to fixed parameter tractability. 

\subsection{Treewidth}
\label{sec:finite-state-proof}

We first sketch an argument as to why it is not possible to use Courcelle's Theorem (or any automata methods based on tree decompositions) to produce a fixed-parameter tractable algorithm for leveled planarity with respect to treewidth. Consider the family of graphs depicted in~\cref{fig:two-paths}. These graphs have bounded treewidth (in fact pathwidth at most $12$) and are leveled planar precisely when $p$ equals $q$. However, since $p$ and $q$ are unbounded it is necessary to carry more than a finite amount of state between bags in a treewidth decomposition when parsing the decomposition. Thus, the decompositions corresponding to leveled planar graphs cannot be recognized by automata or methods using automata such as Courcelle's Theorem. We now make this intuitive argument formal using the Myhill-Nerode Theorem for tree automata.

\begin{figure}[t]
\centering
\includegraphics[width=0.65\textwidth]{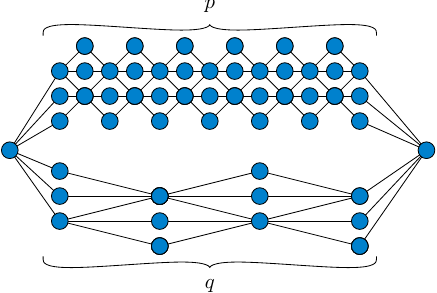}
\caption{A family of graphs with bounded treewidth demonstrating that the family of leveled planar graphs is not finite state.}
\label{fig:two-paths}
\end{figure}

In order to avoid set-theoretic difficulties we consider only finite graphs whose vertices are drawn from a fixed countable set; this involves no loss of generality. Following Downey and Fellows~\cite{FPT-Book}, we define a \emph{$t$-boundaried graph} to be a graph $G$ with $t$ designated \emph{boundary} vertices labeled $1,2,\ldots, t$. Given two $t$-boundaried graphs $G_1$ and $G_2$ we define their \emph{gluing} $G_1 \oplus G_2$ by identifying each boundary vertex of $G_1$ with the boundary vertex of $G_2$ having the same label.

An $n$-ary \emph{$t$-boundaried operator} $\otimes$ consists of a $t$-boundaried graph $G_\otimes = (V_\otimes, E_\otimes)$ and injections $f_i : \{1,\ldots, t\} \to V_\otimes$ for $1 \leq i \leq n$. Then for $t$-boundaried graphs $G_1,\ldots,G_n$ we define the $t$-boundaried graph $G_1 \otimes \cdots \otimes G_n$ by gluing each $G_i$ to $G_\otimes$ after applying $f_i$ to the boundary labels of $G_\otimes$. After the gluing the labels of $G_i$ are forgotten.

It can be shown that there exists a \emph{standard set} of $t$-boundaried operators on $t$-boundaried graphs that can be used to generate the set of all graphs of treewidth~$t$. Furthermore, it is possible to convert (in linear time) a tree decomposition of width $t$ into a parse tree that uses these standard operators; see Theorem~12.7.1 in \cite{FPT-Book}. Define $\mathcal{U}_t$ to be the set of $t$-boundaried graphs obtained by parse trees, using these standard operators. Given a family of graphs $F$, we define the equivalence relation $\sim_F$ on $\mathcal{U}_t$, such that $G_1 \sim_F G_2$ if and only if for all $H \in \mathcal{U}_t$, we have $G_1 \oplus H \in F \Leftrightarrow G_2 \oplus H \in F$.

A family of graphs $F$ is said to be \emph{$t$-finite state} if the family of parse trees for graphs in $F_t = F \cap \mathcal{U}_t$ is $t$-finite state. Equivalently, such a family of parse trees may be recognized by a finite tree automaton. We can now state the analog of the Myhill--Nerode Theorem (characterizing recognizability of sets of strings by finite state machines) for treewidth $t$ graphs in place of strings and finite tree automata in place of finite state machines.

\begin{theorem}[Theorem 12.7.2 of \cite{FPT-Book}]\label{thm:myhill-nerode}
A family of graphs $F$ is $t$-finite state if and only if $\sim_{F}$ has 
finitely many equivalence classes over $\mathcal{U}_t$.
\end{theorem}

\begin{figure}[t]
\centering
\includegraphics[width=0.95\textwidth]{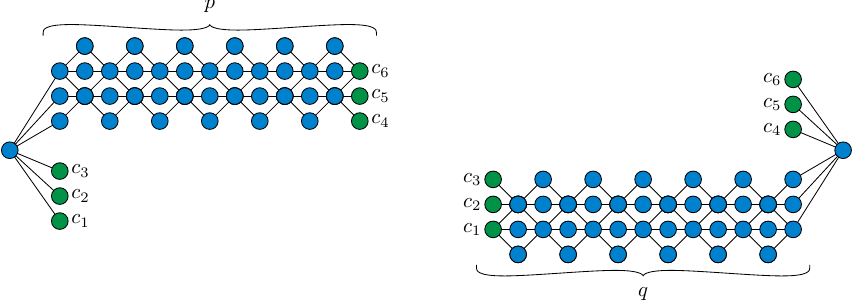}
\caption{The $6$-boundaried graphs $U_p$ (left) and $L_q$ (right) from the proof of~\cref{thm:finite-state}.}
\label{fig:u_l_graphs}
\end{figure}

As we now show, leveled planarity is not $t$-finite state when $t$ is sufficiently large.

\begin{theorem}\label{thm:finite-state}
For all $t \geq 6$,
the family of leveled planar graphs is not $t$-finite state.
\end{theorem}
\begin{proof}
Let $F$ be the family of leveled planar graphs. It suffices to prove the theorem in the case when $t = 6$. Consider the $6$-boundaried graphs $U_p$ and $L_q$ shown in~\cref{fig:u_l_graphs}, and observe that $U_p \oplus L_q$ is leveled planar if and only if $p = q$. So $U_p \sim_F U_{\ell}$ if and only if $p = \ell$, which implies that $\sim_{F_6}$ does not have finitely many equivalence classes, and that in turn $F$ is not $6$-finite state by~\cref{thm:myhill-nerode}.
\end{proof}

\cref{thm:finite-state} implies that (when $t\ge 6$) the parse trees of leveled planar graphs with treewidth $t$ are not recognizable by tree automata. Therefore automata-based methods such as Courcelle's Theorem cannot be used to show leveled planarity to be fixed-parameter tractable with respect to treewidth. In particular, leveled planarity cannot be expressed using the forms of monadic second-order graph logic to which Courcelle's Theorem applies.

\subsection{Tree-depth}

The \emph{tree-depth} of a graph $G$ is the minimum height of a forest of rooted trees on the same vertex set as $G$ such that edges in $G$ only go from ancestors to descendants in the forest. It is bounded by pathwidth, and therefore by track-number: $\text{track-number}(G)\leq\text{pathwidth}(G)+1\leq\text{tree-depth}(G)$; see \cite{DujMorWoo-SJC-05,Sparsity-Book}.

\begin{theorem}
Computing the track-number or layered pathwidth of a graph $G$ is non-uniformly fixed-parameter linear in the tree-depth of $G$.
\end{theorem}

\begin{proof}
Track-number and layered pathwidth are both monotone (cannot increase) under taking induced subgraphs. The graphs with tree-depth bounded by a constant are well-quasi-ordered under taking induced subgraphs and so for any fixed bound on tree-depth and either track-number or layered pathwidth there exist only finitely many forbidden induced subgraphs~\cite{Sparsity-Book}. Since the track-number and pathwidth are both bounded by the tree-depth, the same is true for any fixed bound on tree-depth, regardless of track-number or layered pathwidth.

 Because induced subgraph testing is linear time for graphs with tree-depth bounded by a fixed number $d$, we can for each $t \leq d$ test if the graph has any of the forbidden induced subgraphs to track-number $t$ each in linear time~\cite{Sparsity-Book}.
\end{proof}

However, this argument does not tell us how to find the set of forbidden subgraphs, nor what the dependence of the time bound on the tree-depth is. It would be of interest to replace this existence proof with a more constructive algorithm.

\subsection{Almost-trees}

The \emph{cyclomatic number} (also called \emph{circuit rank}) of a graph is defined to be $r = m - n + c$ where $c$ is the number of connected components in an $n$-vertex $m$-edge graph. We say that a graph $G$ is a \emph{$k$-almost tree} if every biconnected component of $G$ has cyclomatic number at most $k$. The problems of $1$-page and $2$-page crossing minimization and testing $1$-planarity were shown to be fixed-parameter tractable with respect to the $k$-almost tree parameter, via the kernelization method~\cite{BanCabEpp-WADS-2013,BanEppSim-GD-13}.

In these previous papers,
the ``standard kernelization'' used for a $k$-almost tree $G$ is constructed by first iteratively removing degree one vertices until no more remain, leaving what is called the \emph{2-core} of $G$. The $2$-core consists of vertices of degree greater than two and paths of degree two vertices connecting these high degree vertices. The paths of degree two vertices are then shortened to a maximum length whose value is a function of~$k$, with a precise form that depends on the specific problem.

However, this kernelization cannot be used to produce a fixed-parameter tractable algorithm for deciding leveled planarity. To see this, consider the graph in \cref{fig:almost-tree-example}, constructed by drawing $K_{2,3}$ in the plane, and replacing each of the three vertices with paths of $k$ vertices, and then rooting a complete binary tree of depth $\ell$ at one of the vertices of each of these paths. We note that, as complete binary trees have unbounded pathwidth, they also require an unbounded number of layers (depending on $\ell$) in any leveled planar drawing. Additionally, depending on the planar embedding chosen for this graph, at most two of its three trees can be drawn on the outside face.  So this graph is leveled planar precisely when $\ell$ is  small enough for the remaining tree $T_\ell$ to be drawn within one of the four bounded faces of the drawing; that is, the leveled planarity of the graph depends on the relationship between $k$ and $\ell$. Since this relationship is not preserved in the kernelization it cannot be used to produce a fixed-parameter tractable algorithm for leveled planarity.

\begin{figure}[t]
\centering
\includegraphics[height=70mm]{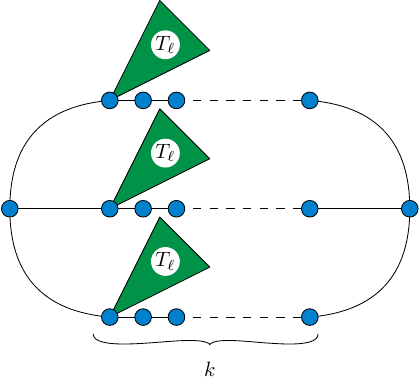}
\caption{A family of $2$-almost trees for which the standard kernelization cannot decide leveled planarity. The subgraphs $T_\ell$ are complete binary trees of depth $\ell$.}
\label{fig:almost-tree-example}
\end{figure}

\newcommand{\MR}[1]{\MRhref{#1}{MR:#1}} 

\providecommand{\bysame}{\leavevmode\hbox to3em{\hrulefill}\thinspace}
\providecommand{\MR}{\relax\ifhmode\unskip\space\fi MR }
\providecommand{\MRhref}[2]{%
  \href{http://www.ams.org/mathscinet-getitem?mr=#1}{#2}%
}
\providecommand{\href}[2]{#2}

\end{document}